\documentclass[reqno, 12pt]{article} 

\pdfoutput=1

\usepackage{enumerate}
\usepackage{latexsym}
\usepackage[centertags]{amsmath}
\usepackage{amsfonts}
\usepackage{amssymb}
\usepackage{amsthm}
\usepackage{newlfont}
\usepackage{graphics}
\usepackage{color}
\usepackage{float}
\usepackage{diagbox}
\usepackage{hyperref}
\usepackage{longtable}
\usepackage{rotating}
\usepackage{multirow}
\usepackage{extarrows}
\usepackage[sort,compress,numbers]{natbib}
\usepackage[utf8]{inputenc}

\newtheorem{thm}{Theorem}[section]
\newtheorem{cor}[thm]{Corollary}
\newtheorem{lem}[thm]{Lemma}

\theoremstyle{mydefinition}

\theoremstyle{myremark}

\allowdisplaybreaks[4]

\def\Res{\mathop{\mathrm{Res}}}

\title{Three Simple Reduction Formulas for the Denumerant Functions}

\author{Feihu Liu$^{1}$, Guoce Xin$^{2,}$\thanks{This work is partially supported by the National Natural Science Foundation of China (No.12071311).},\ and Chen Zhang$^{3}$
\\[2mm]
{\small $^{1, 2, 3}$ School of Mathematical Sciences,}\\[-0.8ex]
{\small Capital Normal University, Beijing, 100048, P.R.~China}\\
{\small $^1$ Email address: liufeihu7476@163.com}\\
{\small $^2$ Email address: guoce\_xin@163.com}\\
{\small $^3$ Email address: ch\_enz@163.com}
}

\date{April 22, 2024}

\begin{document}

\maketitle

\begin{abstract}
Let $A$ be a nonempty set of positive integers. The restricted partition function $p_A(n)$ denotes the number of partitions of $n$ with parts in $A$. When the elements in $A$ are pairwise relatively prime positive integers, Ehrhart, Sert\"oz-\"Ozl\"uk, and Brown-Chou-Shiue derived three reduction formulas for $p_A(n)$ for $A$ with three parameters. We extend their findings for general $A$ using the Bernoulli-Barnes polynomials.
\end{abstract}

\noindent
\begin{small}
 \emph{Mathematic subject classification}: Primary 05A17; Secondary 05A15, 11P81.
\end{small}

\noindent
\begin{small}
\end{small}

\section{Introduction}

Let $A=\{a_1,a_2,...,a_k\}$ be a set of positive integers with $k\geq 1$.
Furthermore, let $p_{A}(n)$ denote the number of nonnegative integer solutions to the equation
$$a_1x_1+a_2x_2+\cdots +a_kx_k=n.$$
The $p_A(n)$ is called the \emph{restricted partition function} of the set $A$.
Some scholars also refer to it as \emph{Sylvester's denumerant} \cite{J. J. Sylvester} when $\gcd(A)=1$.

Sylvester \cite{J. J. Sylvester} and Bell \cite{Bell43} proved that $p_A(n)$ is a quasi-polynomial of degree $k-1$, and the period is a common multiple of $a_1,a_2,\ldots, a_k$.
Beck, Gessel, and Komatsu \cite{BeckGesselKomatsu} found an expression for the polynomial part of $p_A(n)$. Nathanson \cite{Nathanson} gave an asymptotic formula of $p_{A}(n)$.
Cimpoeas \cite{Cimpoeas18} proved that the $p_A(n)$ can be reduced to solving a linear congruence formula.
Some relevant references can be found in \cite{BayadBeck,GostGarcia10,Komatsu03,XinZhangDernm}.

For $k=2$, Sert\"oz \cite{Sertoz98} and Tripathi \cite{A.Tripathi} independently obtained an explicit formula for $p_A(n)$.
For $k=3$, Ehrhart \cite{Ehrhart65,Ehrhart66} and Sert\"oz and \"Ozl\"uk \cite{Sertoz86} gave recursive formulae for $p_A(n)$.
In this paper, we first extend the results of Ehrhart \cite{Ehrhart65,Ehrhart66} (the case $k=2,3$ in Theorem \ref{MainThm1}) as follows.
\begin{thm}\label{MainThm1}
Let $A=\{a_1,a_2, \ldots, a_k\}$, where $a_1,a_2,\ldots,a_k$ are pairwise relatively prime positive integers.
Let $n=q\cdot a_1a_2\cdots a_k+r$ with $0\leq r<a_1a_2\cdots a_k$. Then
$$p_{A}(n)=p_{A}(r)+(-1)^k(n-r)\sum_{i=0}^{k-2}\frac{(r-n)^i}{(i+1)!(k-i-2)!}\mathcal{B}_{k-i-2}(-r;a_1,a_2,\ldots,a_k),$$
where $\mathcal{B}_i(x;a_1,a_2,\ldots,a_k)$ is the Bernoulli-Barnes polynomials (defined by Equation \eqref{Bernoulli-Barnes}).
\end{thm}

Secondly, we generalize the results of Sert\"oz and \"Ozl\"uk \cite{Sertoz86} (the case $k=3$ in Theorem \ref{MainThm2}) as follows.
\begin{thm}\label{MainThm2}
Let $A=\{a_1,a_2, \ldots, a_k\}$, where $a_1,a_2,\ldots,a_k$ are pairwise relatively prime positive integers.
Let $1\leq x\leq a_1+a_2+\cdots +a_k-1$. Then
$$p_{A}(a_1a_2\cdots a_k-x)=(-1)^k(a_1a_2\cdots a_k)\sum_{i=0}^{k-2}\frac{(-a_1a_2\cdots a_k)^i}{(i+1)!(k-i-2)!}\mathcal{B}_{k-i-2}(x;a_1,a_2,\ldots,a_k).$$
\end{thm}

Thirdly, we extend the results of Brown, Chou, and Shiue \cite{Brown} (the case $k=3$ and $x=a_1+a_2+a_3$ (and $x=a_1+a_2+a_3+1$) in Theorem \ref{MainThm3}) as follows.
\begin{thm}\label{MainThm3}
Let $A=\{a_1,a_2, \ldots, a_k\}$, where $a_1,a_2,\ldots,a_k$ are pairwise relatively prime positive integers.
Let $a_1+a_2+\cdots +a_k\leq x\leq a_1a_2\cdots a_k$. Then
\begin{align*}
&p_{A}(a_1a_2\cdots a_k-x)+(-1)^kp_{A}(x-a_1-a_2-\cdots-a_k)
\\&\ \ \ \ \ \ \ \ \ \ \ \ \ \ \ \ =(-1)^k(a_1a_2\cdots a_k)\sum_{i=0}^{k-2}\frac{(-a_1a_2\cdots a_k)^i}{(i+1)!(k-i-2)!}\mathcal{B}_{k-i-2}(x;a_1,a_2,\ldots,a_k).
\end{align*}
\end{thm}

This paper is organized as follows.
In Section 2, we introduce some necessary notations and provide the proof of Theorem \ref{MainThm1}.
In Section 3, we give a recursive formula for $p_A(n)-p_A(r)$, where $0\leq r<a_1a_2\cdots a_k$.
Sections 4 and 5 give the proofs of Theorems \ref{MainThm2} and \ref{MainThm3}, respectively.
Throughout this paper, $\mathbb{C}$, $\mathbb{N}$, and $\mathbb{P}$ denote the set of all complex numbers, all nonnegative integers, and all positive integers, respectively.

\section{The Proof of Theorem \ref{MainThm1}}

Before obtaining the main results of this section, we need to introduce some definitions and conclusions.
Let $f(\lambda)$ be a rational function in $\mathbb{C}((\lambda))$.
The $\mathrm{CT}_{\lambda}f(\lambda)$ denotes the constant term of the Laurent series expansion of $f(\lambda)$ at $\lambda=0$.
The $\mathrm{Res}_{\lambda=\lambda_0}f(\lambda)$ denotes the residue of $f(\lambda)$ when expanded as a Laurent series at $\lambda=\lambda_0$. More precisely, we have
$$\mathop{\mathrm{Res}}\limits_{\lambda=\lambda_0}\sum_{i\geq i_0}c_i(\lambda-\lambda_0)^i=c_{-1}.$$

For the denumerant $p_{A}(n)$ with $A=\{a_1,a_2,...,a_k\}$, we have
\begin{align}\label{CTDenumer}
p_{A}(n)=\sum_{x_i\geq 0}\mathop{\mathrm{CT}}\limits_{\lambda}\lambda^{x_1a_1+x_2a_2+\cdots +x_ka_k-n}
=\mathop{\mathrm{CT}}\limits_{\lambda}\frac{\lambda^{-n}}{(1-\lambda^{a_1})(1-\lambda^{a_2})\cdots (1-\lambda^{a_k})}.
\end{align}

\begin{lem}[\cite{Jacobi}]\label{ResMeromor}
Let $c$ be a complex number. Suppose $g(s)$ is holomorphic in a neighborhood of $s=c$ and suppose $f(\lambda)$ is meromorphic in a neighborhood of $\lambda=g(c)$. If $g^{\prime}(c)\neq 0$, then
$$\mathop{\mathrm{Res}}\limits_{\lambda=g(c)}f(\lambda)=\mathop{\mathrm{Res}}\limits_{s=c}f(g(s))g^{\prime}(s).$$
\end{lem}

\begin{lem}\label{ResNonegit}
Let $r_1,r_2, \dots, r_k \in \mathbb{P}$ and $b \le r_1 +r_2+ \cdots + r_k - 1$. Suppose
$$f(z)=\frac{z^{b-1}}{(z-\xi_1)^{r_1}(z-\xi_2)^{r_2}\cdots (z-\xi_k)^{r_k}}.$$
Then
$$\Res_{z=0} f(z) = -\sum_{i=1}^k   \Res_{z=\xi_i} f(z).$$
\end{lem}
\begin{proof}
A well-known result in residue computation asserts that
$$\Res_{z=\infty} f(z) + \Res_{z=0} f(z) + \sum_{i=1}^k   \Res_{z=\xi_i} f(z) = 0.$$
The lemma then follows by showing that $\Res\limits_{z=\infty} f(z) = 0$.
Direct computation gives
\begin{align*}
\Res\limits_{z=\infty} f(z) &=  \Res\limits_{z=0} f(z^{-1}) \cdot (-z^{-2}) \\
 &= \Res\limits_{z=0} \frac{-z^{-b-1}}{(z^{-1}-\xi_1)^{r_1}\cdots (z^{-1}-\xi_k)^{r_k}} 
 = \Res\limits_{z=0} \frac{-z^{r_1 + \cdots + r_k-b-1}}{(1-z \xi_1)^{r_1}\cdots (1- z\xi_k)^{r_k}}. \\
\end{align*}
Since $b \le r_1 + \cdots + r_k - 1$, the expansion of $\frac{-z^{r_1 + \cdots + r_k-b-1}}{(1-z \xi_1)^{r_1}\cdots (1- z\xi_k)^{r_k}}$ is a power series in $z$. Therefore, its residue at $z=0$ is $0$. 
This completes the proof.
\end{proof}

For $a_1,a_2,\ldots,a_k \in \mathbb{P}$, the \emph{Bernoulli-Barnes polynomials} $\mathcal{B}_i(x;a_1,a_2,\ldots,a_k)$
are polynomials in $x$ defined by
\begin{align}\label{Bernoulli-Barnes}
\frac{s^ke^{xs}}{(e^{a_1s}-1)(e^{a_2s}-1)\cdots (e^{a_ks}-1)}
=\sum_{i\geq 0}\mathcal{B}_i(x;a_1,a_2,\ldots,a_k)\frac{s^i}{i!}.
\end{align}

\begin{proof}[Proof of Theorem \ref{MainThm1}]
By Equation \eqref{CTDenumer}, we have
\begin{align*}
p_A(n)-p_A(r)&=\mathop{\mathrm{CT}}\limits_{\lambda}\frac{\lambda^{-n}}{(1-\lambda^{a_1})(1-\lambda^{a_2})\cdots (1-\lambda^{a_k})}
-\mathop{\mathrm{CT}}\limits_{\lambda}\frac{\lambda^{-r}}{(1-\lambda^{a_1})(1-\lambda^{a_2})\cdots (1-\lambda^{a_k})}
\\&=\mathop{\mathrm{CT}}\limits_{\lambda}\frac{\lambda^{-r}(\lambda^{-qa_1a_2\cdots a_k}-1)}{(1-\lambda^{a_1})(1-\lambda^{a_2})\cdots (1-\lambda^{a_k})}.
\end{align*}
For convenience, let
$$F(\lambda)=\frac{\lambda^{-r-1}(\lambda^{-qa_1a_2\cdots a_k}-1)}{(1-\lambda^{a_1})(1-\lambda^{a_2})\cdots (1-\lambda^{a_k})}.$$
Then
\begin{align*}
p_A(n)-p_A(r)=\mathop{\mathrm{CT}}\limits_{\lambda}\lambda F(\lambda)
=\mathop{\mathrm{Res}}\limits_{\lambda=0}F(\lambda)
=-\sum_{\xi}\mathop{\mathrm{Res}}\limits_{\lambda=\xi}F(\lambda),\ \ \ (\text{By Lemma \ref{ResNonegit}}.)
\end{align*}
where $\xi$ ranges over all nonzero poles of $F(\lambda)$. We claim that $\mathrm{Res}_{\lambda=\xi}F(\lambda)=0$ unless $\xi=1$.
Since $a_1,a_2,\ldots,a_k$ are pairwise relatively prime positive integers, each $\xi\neq 1$ 
appears exactly once in the denominator, but the numerator also vanishes at these $\xi$'s. 
Therefore, we obtain
\begin{align}\label{CTFGNeed}
p_A(n)-p_A(r)&=-\mathop{\mathrm{Res}}\limits_{\lambda=1}F(\lambda)
=-\mathop{\mathrm{Res}}\limits_{s=0}F(e^s)e^s\ \ \ \ \ (\text{By Lemma \ref{ResMeromor}}.)\nonumber
\\&=-\mathop{\mathrm{CT}}\limits_{s}F(e^s)e^s s
\\&=-\mathop{\mathrm{CT}}\limits_{s}\frac{e^{-rs}(e^{-qa_1a_2\cdots a_ks}-1)s}{(1-e^{a_1s})(1-e^{a_2s})\cdots (1-e^{a_ks})}\nonumber
\\&=(-1)^kqa_1a_2\cdots a_k \mathop{\mathrm{CT}}\limits_{s}\frac{1}{s^{k-2}}\cdot \frac{e^{-qa_1a_2\cdots a_ks}-1}{-qa_1a_2\cdots a_ks}\cdot \frac{s^ke^{-rs}}{\prod_{i=1}^k(e^{a_is}-1)}\nonumber
\\&=(-1)^kqa_1a_2\cdots a_k [s^{k-2}]\sum_{i\geq 0}\frac{(-qa_1a_2\cdots a_k)^i}{(i+1)!}s^i
\cdot \sum_{j\geq 0}\mathcal{B}_j(-r;a_1,a_2,\ldots,a_k)\frac{s^j}{j!}\nonumber
\\&=(-1)^k(n-r)\sum_{i=0}^{k-2}\frac{(r-n)^i}{(i+1)!(k-i-2)!}\mathcal{B}_{k-i-2}(-r;a_1,a_2,\ldots,a_k).\nonumber
\end{align}
This completes the proof.
\end{proof}

\begin{cor}[\cite{Ehrhart65,Ehrhart66}]\label{Corr23}
Following the notation in Theorem \ref{MainThm1}.
If $A=\{a_1,a_2\}$, then
$$p_A(n)=p_A(r)+\frac{n-r}{a_1a_2}.$$
If $A=\{a_1,a_2,a_3\}$, then
$$p_A(n)=p_A(r)+\frac{q(n+r+a_1+a_2+a_3)}{2}.$$
\end{cor}
\begin{proof}
By
\begin{align*}
\frac{s^2e^{-rs}}{(e^{a_1s}-1)(e^{a_2s}-1)}=\frac{1}{a_1a_2}+o(s)
\end{align*}
and
\begin{align*}
\frac{s^3e^{-rs}}{(e^{a_1s}-1)(e^{a_2s}-1)(e^{a_3s}-1)}=\frac{1}{a_1a_2a_3}-\frac{2r+a_1+a_2+a_3}{2a_1a_2a_3}s+o(s^2),
\end{align*}
the corollary follows from Theorem \ref{MainThm1}.
\end{proof}

\begin{cor}\label{Corr45}
Following the notation in Theorem \ref{MainThm1}.
If $A=\{a_1,a_2,a_3,a_4\}$, then
\begin{align*}
p_A(n)=p_A(r)+\frac{q}{12}\Big(&3(n+r)(a_1+a_2+a_3+a_4)+2(n+r)^2-2nr+(a_1+a_2+a_3+a_4)^2
\\&\ \ +a_1a_2+a_1a_3+a_1a_4+a_2a_3+a_2a_4+a_3a_4\Big).
\end{align*}
If $A=\{a_1,a_2,a_3,a_4,a_5\}$, then
\begin{align*}
p_A(n)=&p_A(r)+\frac{q}{24}\bigg((n+r)(n^2+r^2)+(2n^2+2nr+2r^2)\Big(\sum_{i=1}^5a_i\Big)
+(n+r)\Big(\sum_{i=1}^5a_i^2\Big)
\\&+\sum_{i=1}^5a_i^2\Big(\sum_{j=1}^5a_j-a_i\Big)+3(n+r)\sum_{1\leq i<j\leq 5}a_ia_j+3\sum_{1\leq i<j\leq 5}\frac{a_1a_2a_3a_4a_5}{a_ia_j}\bigg).
\end{align*}
\end{cor}
The proof of Corollary \ref{Corr45} is analogous to that of Corollary \ref{Corr23} and is left to the reader.

\section{A Recursive Formula for $p_A(n)-p_A(r)$}

Readers familiar with symmetric functions may find that the formulas in Corollary \ref{Corr45} are related to the power sum symmetric function. 
This does not occur occasionally. We will describe this connection in general. We will also give a recursive formula for $p_A(n)-p_A(r)$
when $n-r=qa_1\cdots a_k$ as in Section 2.

We need the following definitions.
The \emph{$m$-th power sum symmetric function} is
$$p_m(x_1,x_2,\ldots)=\sum_{i\geq 1}x_i^m.$$
We only use the symmetric functions on a finite number of variables, say $x_1,\dots, x_k$. One can treat $x_i=0$ for $i>k$.
We have $p_m(x_1,x_2,\ldots,x_k)=\sum_{i=1}^kx_i^m$.
The \emph{Bernoulli numbers} $\mathcal{B}_i$ are defined by
$$\frac{s}{e^s-1}=1-\mathcal{B}_1s+\sum_{i\geq 2}\mathcal{B}_i\frac{s^i}{i!}=1-\frac{1}{2}s+\frac{1}{12}s^2-\frac{1}{720}s^4+\cdots.$$
Then
$$\ln \frac{s}{e^s-1}=-\sum_{i\geq 1}\frac{\mathcal{B}_i}{i!\cdot i}s^i=-\frac{1}{2}s-\frac{1}{24}s^2+\frac{1}{2880}s^4+\cdots.$$

By the proof of Theorem \ref{MainThm1}, we have
\begin{align*}
p_A(n)-p_A(r)&=(-1)^kq \mathop{\mathrm{CT}}\limits_{s}\frac{1}{s^{k-2}}\cdot e^{-rs}\cdot \frac{e^{-qa_1a_2\cdots a_ks}-1}{-qa_1a_2\cdots a_ks}\cdot \prod_{i=1}^k\frac{a_is}{e^{a_is}-1}.
\end{align*}
We consider the following formula. The technique for taking logarithms below comes from \cite{XinZhangZhang2024}.
\begin{align*}
h(s)&=\ln \left(e^{-rs}\cdot \frac{e^{-qa_1a_2\cdots a_ks}-1}{-qa_1a_2\cdots a_ks}\cdot \prod_{i=1}^k\frac{a_is}{e^{a_is}-1}\right)
\\&=-rs+\ln \left(\frac{-qa_1a_2\cdots a_ks}{e^{-qa_1a_2\cdots a_ks}-1}\right)^{-1}+\sum_{i=1}^k\ln \frac{a_is}{e^{a_is}-1}
\\&=-rs+\sum_{i\geq 1}\frac{\mathcal{B}_i}{i!\cdot i}(-qa_1a_2\cdots a_k)^is^i+\sum_{j\geq 1}\frac{-\mathcal{B}_jp_j(a_1,a_2,\ldots,a_k)}{j!\cdot j}s^j
\\&=-rs+\sum_{i\geq 1}\frac{\mathcal{B}_i}{i!\cdot i}((r-n)^i-p_i(a_1,a_2,\ldots,a_k))s^i.
\end{align*}
Then
\begin{align*}
p_A(n)-p_A(r)=(-1)^kq \mathop{\mathrm{CT}}\limits_{s}\frac{1}{s^{k-2}}\cdot e^{h(s)}=(-1)^kq[s^{k-2}]e^{h(s)}.
\end{align*}
Let $f(s)=e^{h(s)}=\sum_{i\geq 0}f_is^i$.
By
$$f^{\prime}(s)=e^{h(s)}\cdot h^{\prime}(s)=f(s)\cdot h^{\prime}(s),$$
we have
$$\sum_{i\geq 0}if_is^{i-1}=\sum_{i\geq 0}(f_is^i)\cdot \sum_{i\geq 0}(h_i^{\prime}s^i),$$
that is
$$f_0=1,\ \ f_i=\frac{1}{i}\cdot \sum_{j=1}^{i}f_{i-j}h_{j-1}^{\prime},\ \ (i\geq 1).$$
Therefore, we have
$$p_A(n)-p_A(r)=(-1)^kq\cdot f_{k-2}.$$

We summarize the above discussion as follows.
\begin{thm}
Let $A=\{a_1,a_2, \ldots, a_k\}$, where $a_1,a_2,\ldots,a_k$ are pairwise relatively prime positive integers.
Let $n=q\cdot a_1a_2\cdots a_k+r$ with $0\leq r<a_1a_2\cdots a_k$.
Suppose
$$h(s)=-rs+\sum_{i\geq 1}\frac{\mathcal{B}_i}{i!\cdot i}((r-n)^i-p_i(a_1,a_2,\ldots,a_k))s^i,\ \ \ \ h^{\prime}(s)=\sum_{i\geq 0}h^{\prime}_is^i.$$
Then
$$p_{A}(n)=p_{A}(r)+(-1)^kq\cdot f_{k-2},$$
where $f_i$ can be recursively obtained by
$$f_0=1,\ \ \ \ f_i=\frac{1}{i}\cdot \sum_{j=1}^{i}f_{i-j}h_{j-1}^{\prime},\ \ (i\geq 1).$$
\end{thm}

\section{The Proof of Theorem \ref{MainThm2}}

\begin{proof}[Proof of Theorem \ref{MainThm2}]
By Equation \eqref{CTDenumer}, we have
\begin{align*}
p_A(a_1a_2\cdots a_k-x)&=\mathop{\mathrm{CT}}\limits_{\lambda}\frac{\lambda^{-a_1a_2\cdots a_k+x}}{(1-\lambda^{a_1})(1-\lambda^{a_2})\cdots (1-\lambda^{a_k})}.
\end{align*}
The rational function
$$\frac{\lambda^{x}}{(1-\lambda^{a_1})(1-\lambda^{a_2})\cdots (1-\lambda^{a_k})}$$
is a proper rational function since $1\leq x\leq a_1+a_2+\cdots +a_k-1$. Obviously, its constant term is $0$.
We have
\begin{align*}
p_A(a_1a_2\cdots a_k-x)&=\mathop{\mathrm{CT}}\limits_{\lambda}\frac{\lambda^{-a_1a_2\cdots a_k+x}-\lambda^{x}}{(1-\lambda^{a_1})(1-\lambda^{a_2})\cdots (1-\lambda^{a_k})}
=\mathop{\mathrm{CT}}\limits_{\lambda}\frac{\lambda^{x}(\lambda^{-a_1a_2\cdots a_k}-1)}{(1-\lambda^{a_1})(1-\lambda^{a_2})\cdots (1-\lambda^{a_k})}.
\end{align*}
The remainder of the argument is analogous to that in Theorem \ref{MainThm1} and is left to the reader.
\end{proof}

Similar to Corollaries \ref{Corr23} and \ref{Corr45}, we can obtain the following three corollaries. We omit the proofs.
\begin{cor}[\cite{Sertoz98}]
Let $A=\{a_1,a_2\}$ with $\gcd(A)=1$. Let $1\leq x\leq a_1+a_2-1$. Then
$$p_A(a_1a_2-x)=1.$$
\end{cor}

\begin{cor}[\cite{Sertoz86}]
Let $A=\{a_1,a_2, a_3\}$, where $a_1,a_2,a_3$ are pairwise relatively prime positive integers.
Let $1\leq x\leq a_1+a_2+a_3-1$. Then
$$p_A(a_1a_2a_3-x)=\frac{a_1a_2a_3+a_1+a_2+a_3}{2}-x.$$
\end{cor}

\begin{cor}
Following the notation in Theorem \ref{MainThm2}.
If $A=\{a_1,a_2,a_3,a_4\}$, then
\begin{align*}
p_A(a_1a_2a_3a_4-x)=\frac{1}{12}\Big(&3(n-x)(a_1+a_2+a_3+a_4)+2(n-x)^2+2nx+(a_1+a_2+a_3+a_4)^2
\\&\ \ +a_1a_2+a_1a_3+a_1a_4+a_2a_3+a_2a_4+a_3a_4\Big),
\end{align*}
where $n=a_1a_2a_3a_4-x$.

If $A=\{a_1,a_2,a_3,a_4,a_5\}$, then
\begin{align*}
p_A(a_1a_2a_3&a_4a_5-x)=\frac{1}{24}\bigg((n-x)(n^2+x^2)+(2n^2-2nx+2x^2)\Big(\sum_{i=1}^5a_i\Big)
+(n-x)\Big(\sum_{i=1}^5a_i^2\Big)
\\&+\sum_{i=1}^5a_i^2\Big(\sum_{j=1}^5a_j-a_i\Big)+3(n-x)\sum_{1\leq i<j\leq 5}a_ia_j+3\sum_{1\leq i<j\leq 5}\frac{a_1a_2a_3a_4a_5}{a_ia_j}\bigg),
\end{align*}
where $n=a_1a_2a_3a_4a_5-x$.
\end{cor}

\section{The Proof of Theorem \ref{MainThm3}}

\begin{proof}[Proof of Theorem \ref{MainThm3}]
By Equation \eqref{CTDenumer}, we have
\begin{align*}
p_A(a_1a_2\cdots a_k-x)&=\mathop{\mathrm{CT}}\limits_{\lambda}\frac{\lambda^{-a_1a_2\cdots a_k+x}}{(1-\lambda^{a_1})(1-\lambda^{a_2})\cdots (1-\lambda^{a_k})}.
\end{align*}
Since $a_1+a_2+\cdots +a_k\leq x\leq a_1a_2\cdots a_k$, we have 
\begin{align*}
\mathop{\mathrm{CT}}\limits_{\lambda}\frac{\lambda^{x}}{(1-\lambda^{a_1})(1-\lambda^{a_2})\cdots (1-\lambda^{a_k})}=0.
\end{align*}
We have
\begin{align*}
p_A(a_1a_2\cdots a_k-x)&=\mathop{\mathrm{CT}}\limits_{\lambda}\frac{\lambda^{-a_1a_2\cdots a_k+x}-\lambda^{x}}{(1-\lambda^{a_1})(1-\lambda^{a_2})\cdots (1-\lambda^{a_k})}
=\mathop{\mathrm{CT}}\limits_{\lambda}\frac{\lambda^{x}(\lambda^{-a_1a_2\cdots a_k}-1)}{(1-\lambda^{a_1})(1-\lambda^{a_2})\cdots (1-\lambda^{a_k})}.
\end{align*}
Let
$$G(\lambda)=\frac{\lambda^{x-1}(\lambda^{-a_1a_2\cdots a_k}-1)}{(1-\lambda^{a_1})(1-\lambda^{a_2})\cdots (1-\lambda^{a_k})}.$$
Then
\begin{align*}
p_A(a_1a_2\cdots a_k-x)=\mathop{\mathrm{CT}}\limits_{\lambda}\lambda G(\lambda)
=\mathop{\mathrm{Res}}\limits_{\lambda=0}G(\lambda)
=-\sum_{\xi}\mathop{\mathrm{Res}}\limits_{\lambda=\xi}G(\lambda)-\mathop{\mathrm{Res}}\limits_{\lambda=\infty}G(\lambda),
\end{align*}
where $\xi$ ranges over all nonzero poles of $G(\lambda)$. 
Similar to the proof of Theorem \ref{MainThm1}, we have 
$$p_A(a_1a_2\cdots a_k-x)=-\mathop{\mathrm{Res}}\limits_{\lambda=1}G(\lambda)-\mathop{\mathrm{Res}}\limits_{\lambda=\infty}G(\lambda).$$
By $a_1+a_2+\cdots +a_k\leq x\leq a_1a_2\cdots a_k$, we obtain
\begin{align}\label{RRES=0}
\mathop{\mathrm{Res}}\limits_{\lambda=\infty}\frac{\lambda^{x-1-a_1a_2\cdots a_k}}{(1-\lambda^{a_1})(1-\lambda^{a_2})\cdots (1-\lambda^{a_k})}=0.
\end{align}
The proof of Equation \eqref{RRES=0} is similar to Lemma \ref{ResNonegit}.
Let 
$$G_1(\lambda)=\frac{-\lambda^{x-1}}{(1-\lambda^{a_1})(1-\lambda^{a_2})\cdots (1-\lambda^{a_k})}.$$
Then 
\begin{align*}
p_A(a_1a_2\cdots a_k-x)
&=-\mathop{\mathrm{Res}}\limits_{\lambda=1}G(\lambda)-\mathop{\mathrm{Res}}\limits_{\lambda=\infty}G_1(\lambda)
=-\mathop{\mathrm{Res}}\limits_{s=0}G(e^s)\cdot e^s-\mathop{\mathrm{Res}}\limits_{\lambda=0}G_1(\lambda^{-1})\cdot \frac{-1}{\lambda^2}
\\&=-\mathop{\mathrm{CT}}\limits_{s}G(e^s)\cdot e^s\cdot s+\mathop{\mathrm{CT}}\limits_{\lambda}G_1(\lambda^{-1})\cdot \frac{1}{\lambda}
\\&=-\mathop{\mathrm{CT}}\limits_{s}G(e^s)\cdot e^s\cdot s+\mathop{\mathrm{CT}}\limits_{\lambda}\frac{(-1)^{k+1}\lambda^{a_1+a_2+\cdots +a_k-x}}{(1-\lambda^{a_1})(1-\lambda^{a_2})\cdots (1-\lambda^{a_k})}
\\&=-\mathop{\mathrm{CT}}\limits_{s}G(e^s)\cdot e^s\cdot s+(-1)^{k+1}p_{A}(x-a_1-a_2-\cdots -a_k).
\end{align*}
The remainder of the argument is analogous to Equation \eqref{CTFGNeed} and is left to the reader.
\end{proof}

\begin{cor}\label{A123K3}
Let $A=\{a_1,a_2, a_3\}$, where $a_1,a_2,a_3$ are pairwise relatively prime positive integers.
Let $a_1+a_2+a_3\leq x\leq a_1a_2a_3$. Then
$$p_A(a_1a_2a_3-x)-p_{A}(x-a_1-a_2-a_3)=\frac{a_1a_2a_3+a_1+a_2+a_3}{2}-x.$$
\end{cor}

When $x=a_1+a_2+a_3$ (and $x=a_1+a_2+a_3+1$) in Corollary \ref{A123K3}, we obtain the results of Brown, Chou, and Shiue \cite{Brown} as follows: 
$$p_A(a_1a_2a_3-a_1-a_2-a_3)=\frac{a_1a_2a_3-a_1-a_2-a_3}{2}+1,$$
and 
$$p_A(a_1a_2a_3-a_1-a_2-a_3-1)=p_A(1)+\frac{a_1a_2a_3-a_1-a_2-a_3}{2}-1.$$
Note: $p_A(1)=0$ when $a_1,a_2,a_3\geq 2$.


%

\noindent
{\small \textbf{Acknowledgements:}
This work is partially supported by the National Natural Science Foundation of China [12071311].


\begin{thebibliography}{99}

\bibitem{GostGarcia10} F. Aguil\'o-Gost and P. A. Garc\'ia-S\'anchez, \emph{Factoring in embedding dimension three numerical semigroups}, Electron. J. Combin. 17 (2010), \#R138.

\bibitem{BayadBeck} A. Bayad and M. Beck, \emph{Relations for Bernoulli-Barnes numbers and Barnes zeta functions}, Int. J. Number Theory. 10 (2014), 1321--1335.

\bibitem{BeckGesselKomatsu} M. Beck, I. M. Gessel, and T. Komatsu, \emph{The polynomial part of a restricted partition related to the Frobenius problem}, Electron. J. Combin. 8(1) (2001), \#7.

\bibitem{Bell43} E. T. Bell, \emph{Interpolated denumerants and Lambert series}, Am. J. Math. 65 (1943), 382--386.

\bibitem{Brown} T. C. Brown, W. S. Chou, and P. J. Shiue, \emph{On the partition function of a finite set}, Australas. J. Combin. 27 (2003), 193--204.

\bibitem{Cimpoeas18} M. Cimpoeas and F. Nicolae, \emph{On the restricted partition function}, Ramanujan J. 47 (2018), 565--588.

\bibitem{Ehrhart65} E. Ehrhart, \emph{Sur un probl\'eme de g\'eom\'etrie diophantienne lin\'eaire. I. poly\'edreset r\'eseaux}, J. Reine Angew. Math. 226 (1965), 1--29.

\bibitem{Ehrhart66} E. Ehrhart, \emph{Sur un probl\'eme de g\'eom\'etrie diophantienne lin\'eaire. II. syst\'emes diophantiens lin\'eaires}, J. Reine Angew. Math. 227 (1966), 30--54.

\bibitem{Jacobi} C. G. J. Jacobi, \emph{De resolutione aequationum per series infinitas}, J. Reine Angew. Math. 6 (1830), 257--286.

\bibitem{Komatsu03} T. Komatsu, \emph{On the number of solutions of the Diophantine equation of Frobenius-general case}, Math. Commun. 28 (2003), 195--206.

\bibitem{Nathanson} M. B. Nathanson, \emph{Partition with parts in a finite set}, Proc. Amer. Math. Soc. 128 (2000), 1269--1273.

\bibitem{Sertoz98} S. Sert\"oz, \emph{On the number of solutions of a diophantine equation of Frobenius}, Discrete Math. Appl. 8 (1998), 153--162.

\bibitem{Sertoz86} S. Sert\"oz and A. E. \"Ozl\"uk, \emph{On the number of representations of an integer by a linear form}, Istanbul Tek. \"Univ. Fen Fak. Mat. Derg. 50 (1991), 67--77.

\bibitem{J. J. Sylvester} J. J. Sylvester, \emph{On the partition of numbers}, Quart.J. Pure Appl. Math. 1 (1857),  141--152.

\bibitem{A.Tripathi} A. Tripathi, \emph{The number of solutions to $ax+by=n$}, Fibonacci Quart. 38 (2000), 290--293.

\bibitem{XinZhangDernm} G. Xin and C. Zhang, \emph{An algebraic combinatorial approach to Sylvester's denumerant}, arXiv:2312.01569v1. (2023).

\bibitem{XinZhangZhang2024} G. Xin, Y. Zhang, and Z. Zhang, \emph{Fast evaluation of generalized Todd polynomials: applications to MacMahon's partition analysis and integer programming}, arXiv: 2304.13323v2. (2024).


\end{thebibliography}
\end{document}